\documentclass[]{amsart}

\usepackage{amsmath,amsthm,amssymb}
\usepackage{mathrsfs}
\usepackage{mathtools}
\usepackage[all]{xy}
\usepackage{tikz}
\usepackage{diagbox}
\usepackage{hyperref}
\usepackage{tensor}

\newtheorem{theorem}{Theorem}[section]

\newtheorem{lemma}[theorem]{Lemma}

\theoremstyle{definition}
\newtheorem{definition}[theorem]{Definition}
\newtheorem{notation}[theorem]{Notation}
\newtheorem{example}[theorem]{Example}
\newtheorem{remark}[theorem]{Remark}

\title{The fibration sequences of the lattice path operad}

\author[Florian De Leger]{Florian De Leger}

\thanks{The author was supported by RVO:67985840 and Praemium Academiae of Martin Markl}
	
\begin{document}
	
\begin{abstract}
	We introduce a notion of an \emph{operad of complexity $m$}, for $m \geq 1$. Operads of complexity $1$ are monoids in the category of $\mathbb{N}$-indexed collections, with monoidal product given by the Day convolution, and operads of complexity $2$ are non-symmetric operads. In general, we prove that the operad for operads of complexity $m$ is a suboperad of the $m$-th stage filtration of the lattice path operad introduced by Batanin and Berger \cite{BataninBergerLattice}. Finally, we exhibit fibration sequences involving this new notion, extending the results of Turchin \cite{turchin} and Dwyer-Hess \cite{dwyerhess}.
\end{abstract}
	
\maketitle

\tableofcontents

\section{Introduction}

Batanin and Berger introduced the \emph{lattice path operad} \cite[Definition 2.1]{BataninBergerLattice} in order to give a conceptual proof of Deligne's conjecture. They exhibited the existence of a \emph{filtration by complexity} of this coloured operad. For $m=2$, they were able to characterise the $m$-th stage filtration as the operad for \emph{multiplicative non-symmetric operads}. For $m > 2$ however, algebras over the $m$-th stage filtration are not very well understood. In this paper, we will define a notion of an \emph{operad of complexity $m$} and prove that they are algebras for a suboperad of the $m$-th stage filtration.

Let us assume that $m \geq 1$ is fixed. Our notion of an operad of complexity $m$ (see Definition \ref{definitionoperadofcomplexitym}) is based on the following construction. Imagine that we have two strips of length $r$ and $s$ respectively. For an integer $z$, let $\lfloor \frac{z}{2} \rfloor$ be the greatest integer less than or equal to $\frac{z}{2}$, that is, $\lfloor \frac{z}{2} \rfloor$ is equal to $\frac{z}{2}$ if $z$ is even and $\frac{z-1}{2}$ if $z$ is odd. We can make $\lfloor \frac{m}{2} \rfloor$ cuts of length one on the first strip and $\lfloor \frac{m-1}{2} \rfloor$ on the second one, so that $m-1$ cuts in total have been picked. We obtain $m+1$ smaller strips of total length $r+s-m+1$ that we can combine to get a new strip, alternating between the smaller strips coming from the first one and the second one that we had at the beginning. Here is an illustration of this construction when $m=6$:
\[
	\begin{tikzpicture}[scale=.5]
		\draw[fill,red] (0,0) -- (0,1) -- (5,1) -- (5,0) -- (0,0);
		
		\draw (1.2,0) -- (1.2,1);
		\draw (2,0) -- (2,1);
		\draw (3.7,0) -- (3.7,1);
		
		\begin{scope}[shift={(6.5,0)}]
		\draw[fill,green] (0,0) -- (0,1) -- (5,1) -- (5,0) -- (0,0);
		\draw (1.4,0) -- (1.4,1);
		\draw (3.8,0) -- (3.8,1);
		\end{scope}
		
		\draw (13,.5) node{$\to$};
		
		\begin{scope}[shift={(14.5,0)}]
		\draw[fill,red] (0,0) -- (0,1) -- (1.2,1) -- (1.2,0) -- (0,0);
		\draw[fill,green] (1.2,0) -- (1.2,1) -- (2.6,1) -- (2.6,0) -- (1.2,0);
		\draw[fill,red] (2.6,0) -- (2.6,1) -- (3.4,1) -- (3.4,0) -- (2.6,0);
		\draw[fill,green] (3.4,0) -- (3.4,1) -- (5.8,1) -- (5.8,0) -- (3.4,0);
		\draw[fill,red] (5.8,0) -- (5.8,1) -- (7.5,1) -- (7.5,0) -- (5.8,0);
		\draw[fill,green] (7.5,0) -- (7.5,1) -- (8.7,1) -- (8.7,0) -- (7.5,0);
		\draw[fill,red] (8.7,0) -- (8.7,1) -- (10,1) -- (10,0) -- (8.7,0);
		\end{scope}
	\end{tikzpicture}
\]
An \emph{operad of complexity $m$} in a symmetric monoidal category $(\mathcal{E},\otimes,e)$ will be defined as a collection $\mathcal{A}(n)$ of objects in $\mathcal{E}$, for $n \geq 0$, together with a map $\eta: e \to \mathcal{A}(m-1)$ which should be regarded as the unit of the operad. The multiplication of the operad is induced by the construction described above. That is, for $r,s \geq 0$, $i$ an $\lfloor \frac{m}{2} \rfloor$-tuple of elements in $\{1,\ldots,r\}$ and $j$ an $\lfloor \frac{m-1}{2} \rfloor$-tuple of elements in $\{1,\ldots,s\}$, we define a map
\[
\tensor[_i]{\circ}{_j} : \mathcal{A}(r) \otimes \mathcal{A}(s) \to \mathcal{A}(r+s-m+1).
\]
For $m=1$, the construction is the gluing of the two strips. Such concatenation induces the notion of a monoid in the category of $\mathbb{N}$-indexed collections with monoidal product given by the Day convolution
\[
	\mathcal{A} \otimes \mathcal{B}(n) = \coprod_{r+s=n} \mathcal{A}(r) \otimes \mathcal{B}(s).
\]
For $m=2$, the construction consists of making one cut on the first strip and putting in its place the second strip. This induces the notion of a non-symmetric operad as defined by Markl \cite[Definition 11]{markl2008operads}.

Together with the notion of operad of complexity $m$, we will also introduce the notion of a bimodule over $\mathcal{A}$ and $\mathcal{B}$, for any pair $(\mathcal{A},\mathcal{B})$ of operads of complexity $m$. Unfortunately, this notion of bimodules seems hard to describe explicitly, so we will describe the operad which governs them instead. Our goal will be to extend the result of Turchin \cite{turchin} and Dwyer-Hess \cite{dwyerhess}. Their result says that for a multiplicative non-symmetric operad $\mathcal{O}$, there are fibration sequences
\begin{equation}\label{equationfirstfibrationsequence}
	\Omega \mathrm{Map}_{\mathrm{NOp}}(\mathcal{A}ss,u^*(\mathcal{O})) \to \mathrm{Map}_{\mathrm{Bimod}}(\mathcal{A}ss,v^*(\mathcal{O})) \to \mathcal{O}_1
\end{equation}
and
\begin{equation}\label{equationsecondfibrationsequence}
	\Omega \mathrm{Map}_{\mathrm{Bimod}}(\mathcal{A}ss,v^*(\mathcal{O})) \to \mathrm{Map}_{\mathrm{IBimod}}(\mathcal{A}ss,w^*(\mathcal{O})) \to \mathcal{O}_0,
\end{equation}
where $\Omega$ is the \emph{loop space functor} and $\mathrm{Map}_\mathbb{C}(-,-)$ is the \emph{homotopy mapping space} in the category $\mathbb{C}$. $\mathrm{NOp}$, $\mathrm{Bimod}$ and $\mathrm{IBimod}$ are respectively the categories of non-symmetric operads, bimodules over $\mathcal{A}ss$ and infinitesimal bimodules over $\mathcal{A}ss$, where $\mathcal{A}ss$ is the terminal non-symmetric operad. Finally, $u^*$, $v^*$ and $w^*$ are the appropriate forgetful functors. We conjecture that for an algebra $\mathcal{O}$ over the $m$-th stage filtration of the lattice path operad, there is a fibration sequence
\[
	\Omega \mathrm{Map}_{\mathrm{Operad}(m)}(\zeta,u^*(\mathcal{O})) \to \mathrm{Map}_{\mathrm{Bimod}(m)}(\zeta,v^*(\mathcal{O})) \to \mathcal{O}_{m-1},
\]
where $\mathrm{Operad}(m)$ is the category of operads of complexity $m$, $\zeta$ is the terminal object in this category and $\mathrm{Bimod}(m)$ is the category of $\zeta-\zeta$-bimodules. Once again, $u^*$ and $v^*$ are the appropriate forgetful functors. We will actually prove the existence of this fibration sequence for $m \leq 3$ (see Theorem \ref{theorem}). The case $m=3$ is a new result. In the case $m=2$, we recover the fibration sequence \ref{equationfirstfibrationsequence} and in the case $m=1$, we get a fibration sequence very similar to \ref{equationsecondfibrationsequence}, $\mathcal{A}ss$-bimodules being equivalent to operads of complexity $1$ equipped with a compatible structure of a covariant presheaf over the active part of the simplex category $\Delta$. Unfortunately, for $m>3$, the combinatorics seem more involved and we might need to add extra assumptions on $\mathcal{O}$. This could be the topic of a future work.

Let us now describe the plan of this paper. In Section \ref{sectionpreliminaries}, we will recall the definition of the lattice path operad, together with its filtration by complexity. We will also recall the characterisation of the second stage filtration as the operad for multiplicative non-symmetric operads. In Section \ref{sectionoperadcomplexitym}, we will introduce the notion of operads of complexity $m$ and prove that they are algebras for a suboperad of the $m$-th stage filtration of the lattice path operad. In Section \ref{sectionbimodules}, we will introduce the notion of a bimodule, as well as pointed bimodules, over these operads of complexity $m$. Finally, in Section \ref{sectionfibrationsequences}, we will prove our fibration sequence theorem. As in \cite{batanindeleger}, the main ingredient is the cofinality of a map of polynomial monads (see Theorem \ref{theoremcofinality}).

\subsection*{Acknowledgement}

I want to thank Maro\v{s} Grego with whom I started working on this project, and Michael Batanin for our interesting email exchange on this topic.

\section{Preliminaries}\label{sectionpreliminaries}

\subsection{The lattice path operad}

Recall that a \emph{symmetric coloured operad} $\mathcal{O}$ (in $\mathrm{Set}$) is given by a set of \emph{colours} $I$ and for each integer $k \geq 0$ and $k+1$-tuples of colours $(i_1,\ldots,i_k;i)$, a set of \emph{operations} $\mathcal{O}(i_1,\ldots,i_k;i)$, where $i_1,\ldots,i_k$ are called the \emph{sources} and $i$ is called the \emph{target}, together with
\begin{itemize}
	\item for all $i \in I$, an element in $\mathcal{O}(i;i)$ called \emph{identity}
	\item \emph{multiplication} maps
	\[
		\mathcal{O}(i_1,\ldots,i_k;i) \times \mathcal{O}(i_{11},\ldots,i_{1l_1};i_1) \times \ldots \times \mathcal{O}(i_{k1},\ldots,i_{kl_k};i_k) \to \mathcal{O}(i_{11},\ldots,i_{kl_k};i)
	\]
	\item an action of the symmetric group $\Sigma_k$ on $\mathcal{O}(i_1,\ldots,i_k;i)$,
\end{itemize}
satisfying associativity, unit and equivariance axioms. 
%
%
%
For $n > 0$, we write $[n]$ for the category given by the poset $\{0 < \ldots < n\}$. The \emph{funny tensor product} $[m] \otimes [n]$ is the category freely generated by the $(m,n)$-grid. This grid has as vertices the pairs $(i,j)$ with $i \in [m]$ and $j \in [n]$ and an edge from $(i,j)$ to $(i',j')$ if $(i',j')=(i+1,j)$ or $(i',j')=(i,j+1)$. This funny tensor product $\otimes$ actually extends to a symmetric monoidal product on the category of small categories. Recall that a \emph{bipointed category} is a category with two distinguished points and a \emph{bipointed functor} is a functor which preserves those points. The category $[n]$ is bipointed by the pair $(0,n)$ and if $\mathcal{C}$ and $\mathcal{D}$ are two categories bipointed by $(c_0,c_1)$ and $(d_0,d_1)$ respectively, then $\mathcal{C} \otimes \mathcal{D}$ is bipointed by $((c_0,d_0),(c_1,d_1))$.

\begin{definition}
	The \emph{lattice path operad} $\mathcal{L}$ is the operad whose set of colours is the set $\mathbb{N}$ of non-negative integers and for a $k+1$-tuple of non-negative integers $(n_1,\ldots,n_k;n)$, the set of operations $\mathcal{L}(n_1,\ldots,n_k;n)$ is the set of bipointed functors
	\begin{equation}\label{equationbipointedfunctor}
		[n+1] \to [n_1+1] \otimes \ldots \otimes [n_k+1],
	\end{equation}
	the multiplication maps are induced by tensor and composition.
\end{definition}

\begin{example}
	The operations of $\mathcal{L}$ can be drawn as lattice paths. For example, the lattice path
	
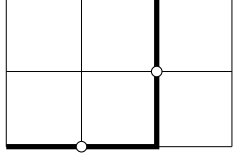
\[
	\begin{tikzpicture}
	\draw[ultra thin] (0,0) -- (0,2);
	\draw[ultra thin] (1,0) -- (1,2);
	\draw[ultra thin] (2,0) -- (2,2);
	\draw[ultra thin] (3,0) -- (3,2);
	\draw[ultra thin] (0,0) -- (3,0);
	\draw[ultra thin] (0,1) -- (3,1);
	\draw[ultra thin] (0,2) -- (3,2);
	\draw[line width=.7mm] (0,0) -- (2,0) -- (2,2) -- (3,2);
	\draw[fill=white] (1,0) circle (2pt);
	\draw[fill=white] (2,1) circle (2pt);
	\end{tikzpicture}
	\]
	represents one of the four bipointed functors $[3] \to [3] \otimes [2]$ sending $1$ to $(1,0)$ and $2$ to $(2,1)$. The lattice path pictured above can also be written as the string $1|12|21$. In general, an operation $x \in \mathcal{L}(n_1,\ldots,n_k;n)$ corresponds to a string containing exactly $n_i+1$ times the integer $i$, for $i=1,\ldots,k$, and exactly $n$ vertical bars. Under this correspondence, the identites are the strings $1|\ldots|1$.
\end{example}

\begin{remark}\label{remarkcompositionlatticepathoperad}
	When we regard the operations of the lattice path operad as strings, the composition is given by replacing each integer of the first string by an appropriate substring, then renumbering, as it was described in \cite[Section 2.2]{BataninBergerLattice}. Here is an example to illustrate:
	\[
		\begin{tikzpicture}
			\draw (0,0) node{$1|12|21$};
			\draw (2,0) node{$213|13|23$};
			\draw (4.2,0) node{$122|211$};
			\draw (5.8,0) node{$\mapsto$};
			\draw (8.1,0) node{$213|13455|54423$};

			\draw (1.28,-.12) -- (1.28,-.17) -- (1.8,-.17) -- (1.8,-.12);
			\draw[->] (1.54,-.17) -- (1.54,-.27) -- (-.46,-.27) -- (-.46,-.12);

			\draw (1.92,-.12) -- (1.92,-.17) -- (2.25,-.17) -- (2.25,-.12);
			\draw[->] (2.085,-.17) -- (2.085,-.37) -- (-.18,-.37) -- (-.18,-.12);

			\draw (2.37,-.12) -- (2.37,-.17) -- (2.7,-.17) -- (2.7,-.12);
			\draw[->] (2.535,-.17) -- (2.535,-.47) -- (.445,-.47) -- (.445,-.12);
			
			\draw (3.65,.16) -- (3.65,.21) -- (4.14,.21) -- (4.14,.16);
			\draw[->] (3.895,.21) -- (3.895,.31) -- (0,.31) -- (0,.19);
			
			\draw (4.25,.16) -- (4.25,.21) -- (4.74,.21) -- (4.74,.16);
			\draw[->] (4.495,.21) -- (4.495,.41) -- (.28,.41) -- (.28,.19);
		\end{tikzpicture}
	\]
\end{remark}


\subsection{Filtration by complexity}


For $1 \leq i < j \leq k$, there are canonical projection functors $[n_1+1] \otimes \ldots \otimes [n_k+1] \to [n_i+1] \otimes [n_j+1]$. Together with the unique functor $[1] \to [n+1]$, they induce maps
\[
	\phi_{ij}: \mathcal{L}(n_1,\ldots,n_k) \to \mathcal{L}(n_i,n_j;0).
\]
We recall the filtration by complexity of the lattice path operad \cite[Definition 2.10]{BataninBergerLattice}.

\begin{definition}
	The complexity index of $x \in \mathcal{L}(n_1,\ldots,n_k;n)$ is defined by
	\[
		c(x) = \max_{1 \leq i < j \leq k} c_{ij}(x),
	\]
	where $c_{ij}(x)$ is the number of corners in the lattice path $\phi_{ij}(x)$.
\end{definition}

\begin{definition}
	The $m$-th stage filtration $\mathcal{L}_m$ is the suboperad of $\mathcal{L}$ of operations of complexity index at most $m$.
\end{definition}

\subsection{Characterisation of the second stage filtration}

Recall that an algebra $\mathcal{A}$ over a symmetric coloured operad $\mathcal{O}$ can be defined in any symmetric monoidal category $(\mathcal{E},\otimes,e)$. It is given by an object $\mathcal{A}(i)$ for each colour $i$ of $\mathcal{O}$ together with a map
\[
	\mathcal{A}(i_1) \otimes \ldots \otimes \mathcal{A}(i_k) \to \mathcal{A}(i)
\]
for each operation in $\mathcal{O}(i_1,\ldots,i_k;i)$, satisfying associativity, unit and equivariance axioms. Also recall that a \emph{non-symmetric operad} is a one-coloured operad without the action of the symmetric group. We therefore get multiplication maps
\[
	\mathcal{A}(k) \otimes \mathcal{A}(l_1) \otimes \ldots \otimes \mathcal{A}(l_k) \to \mathcal{A}(l_1+\ldots+l_k),
\]
where $\mathcal{A}(n) := \mathcal{A}(\underbrace{*,\ldots,*}_n;*)$. It is called \emph{multiplicative} when it is equipped with an operad map from the non-symmetric operad $\mathcal{A}ss$ given by $\mathcal{A}ss(n) = e$ for all $n \geq 0$.
%

\begin{remark}\label{remarkcorrespondence}
	It was proved \cite[Proposition 2.14]{BataninBergerLattice} that algebras of $\mathcal{L}_2$ are multiplicative non-symmetric operads. The argument is that there is a one-to-one correspondence between the operations of $\mathcal{L}_2$ and planar trees with black and white vertices, with the extra condition that there are no pairs of adjacent black vertices and no unary black vertices, as illustrated in the following picture:
	
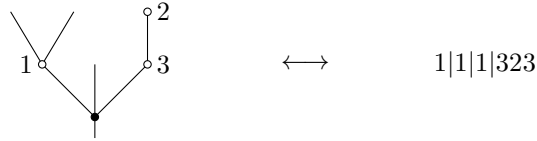
\[
	\begin{tikzpicture}[scale=.7]
	\draw (0,-.4) -- (0,0) -- (-1,1) -- (-1.6,2);
	\draw (-1,1) -- (-.4,2);
	\draw (0,1) -- (0,0) -- (1,1) -- (1,2);
	\draw[fill=white] (-1,1) circle (2pt) node[left]{$1$};
	\draw[fill=white] (1,1) circle (2pt) node[right]{$3$};
	\draw[fill=white] (1,2) circle (2pt) node[right]{$2$};
	\draw[fill] (0,0) circle (2pt);
	
	\draw (4,1) node{$\longleftrightarrow$};
	
	\draw (7.4,1) node{$1|1|1|323$};
	\end{tikzpicture}
	\]
\end{remark}

\section{Operads of complexity $m$}\label{sectionoperadcomplexitym}


\subsection{Definition of the notion}

From now, we will assume that $m \geq 1$ is fixed.

\begin{notation}
	For $r,s \geq 0$, we will write
	\[
		r \vee s = r+s-m+1.
	\]
\end{notation}

For an integer $z$, let $\lfloor \frac{z}{2} \rfloor$ be the greatest integer less than or equal to $\frac{z}{2}$, that is, $\lfloor \frac{z}{2} \rfloor$ is equal to $\frac{z}{2}$ if $z$ is even and $\frac{z-1}{2}$ if $z$ is odd. The following definition will be needed to describe the associativity axioms for operads of complexity $m$.

\begin{definition}
	Let $r,s \geq 0$, $i$ an $\lfloor \frac{m}{2} \rfloor$-tuple of elements in $\{1,\ldots,r\}$, $j$ and $j'$ respectively $\lfloor \frac{m-1}{2} \rfloor$-tuple and $\lfloor \frac{m}{2} \rfloor$-tuple of elements in $\{1,\ldots,s\}$. We define the $\lfloor \frac{m}{2} \rfloor$-tuple $i'$ of elements in $\{1,\ldots,r \vee s\}$ as follows. For $\alpha$ an integer between $1$ and $\lfloor \frac{m}{2} \rfloor$, let $n$ be the number elements in $j$ which are less than the $\alpha$-th element of $j'$. We define
	\[
		i'_\alpha = i_{n+1}+j'_\alpha-2n+1,
	\]
	where $i_{n+1}$ is the $n+1$-th element in $i$ if $n < \lfloor \frac{m}{2} \rfloor$, $r$ otherwise.
	
	Here is a picture for $m=5$ to illustrate:
	
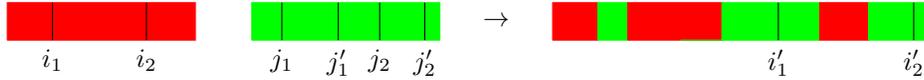
\[
	\begin{tikzpicture}[scale=.5]
	\draw[fill,red] (0,0) -- (0,1) -- (5,1) -- (5,0) -- (0,0);
	
	\draw (1.2,1) -- (1.2,0) node[below]{$i_1$};
	\draw (3.7,1) -- (3.7,0) node[below]{$i_2$};
	
	\begin{scope}[shift={(6.5,0)}]
	\draw[fill,green] (0,0) -- (0,1) -- (5,1) -- (5,0) -- (0,0);
	\draw (.8,1) -- (.8,0) node[below]{$j_1$};
	\draw (2.3,1) -- (2.3,0) node[below]{$j'_1$};
	\draw (3.4,1) -- (3.4,0) node[below]{$j_2$};
	\draw (4.6,1) -- (4.6,0) node[below]{$j'_2$};
	\end{scope}
	
	\draw (13,.5) node{$\to$};
	
	\begin{scope}[shift={(14.5,0)}]
	\draw[fill,red] (0,0) -- (0,1) -- (1.2,1) -- (1.2,0) -- (0,0);
	\draw[fill,green] (1.2,0) -- (1.2,1) -- (2,1) -- (2,0) -- (1.2,0);
	\draw[fill,red] (2,0) -- (2,1) -- (4.5,1) -- (4.5,0) -- (2,0);
	\draw[fill,green] (4.5,0) -- (4.5,1) -- (7.1,1) -- (7.1,0) -- (3.4,0);
	\draw[fill,red] (7.1,0) -- (7.1,1) -- (8.4,1) -- (8.4,0) -- (7.1,0);
	\draw[fill,green] (8.4,0) -- (8.4,1) -- (10,1) -- (10,0) -- (8.4,0);
	\draw (6,1) -- (6,0) node[below]{$i'_1$};
	\draw (9.6,1) -- (9.6,0) node[below]{$i'_2$};
	\end{scope}
	\end{tikzpicture}
	\]
	The red strip represents the set $\{1,\ldots,r\}$ with $i$ given by the pair of elements $(i_1,i_2)$ in this set. The green strip represents the set $\{1,\ldots,s\}$ with $j$ and $j'$ given by the two pairs of elements $(j_1,j_2)$ and $(j'_1,j'_2)$ in this set. The strip on the right represents the set $\{1,\ldots,r \vee s\}$, whose elements are seen as elements of $\{1,\ldots,r\}$ minus the elements of $i$ or elements of $\{1,\ldots,s\}$ minus the elements of $j$. Then $i'_1$ and $i'_2$ are the elements corresponding to $j'_1$ and $j'_2$.
	
	Similarly, let $i$ an $\lfloor \frac{m}{2} \rfloor$-tuple and $i'$ an $\lfloor \frac{m}{2} \rfloor$-tuple (resp. $\lfloor \frac{m-1}{2} \rfloor$-tuple) of elements in $\{1,\ldots,r\}$, and $j$ an $\lfloor \frac{m-1}{2} \rfloor$-tuple of elements in $\{1,\ldots,s\}$. We define the $\lfloor \frac{m}{2} \rfloor$-tuple (resp. $\lfloor \frac{m-1}{2} \rfloor$-tuple) $j'$ of elements in $\{1,\ldots,r \vee s\}$ as follows. For $\alpha$ an integer between $1$ and $\lfloor \frac{m}{2} \rfloor$ (resp. $\lfloor \frac{m-1}{2} \rfloor$), let $n$ be the number of elements in $i$ which are less than the $\alpha$-th element of $i'$. We define
	\[
		j'_\alpha = j_n + i'_\alpha - 2n,
	\]
	where $j_n$ is the $n$-th element in $j_n$ if $1 \leq n \leq \lfloor \frac{m-1}{2} \rfloor$, $0$ if $n=0$ and $s$ if $n > \lfloor \frac{m-1}{2} \rfloor$.
\end{definition}

\begin{definition}\label{definitionoperadofcomplexitym}
	An \emph{operad $\mathcal{A}$ of complexity $m$} in a symmetric monoidal category $(\mathcal{E},\otimes,e)$ is given by a collection $\mathcal{A}(n)$ of objects in $\mathcal{E}$, for $n \geq 0$ together with
	\begin{itemize}
		\item a map
		\begin{equation}\label{equationunitmap}
			\eta: e \to A(m-1),
		\end{equation}
		
		\item for $r,s \geq 0$, $i$ an $\lfloor \frac{m}{2} \rfloor$-tuple of elements in $\{1,\ldots,r\}$ and $j$ an $\lfloor \frac{m-1}{2} \rfloor$-tuple of elements in $\{1,\ldots,s\}$, a map
		\begin{equation}\label{equationmultiplicationmap}
			\tensor[_i]{\circ}{_j} : \mathcal{A}(r) \otimes \mathcal{A}(s) \to \mathcal{A}(r \vee s)
		\end{equation}
	\end{itemize}
	such that
	\begin{itemize}
		\item if $i=\left(1,3,\ldots,2\lfloor \frac{m}{2} \rfloor-1\right)$, for any $s \geq 0$ and $j$ an $\lfloor \frac{m-1}{2} \rfloor$-tuple of elements in $\{1,\ldots,s\}$, the following composite is the identity
		\[
			 \mathcal{A}(s) \xrightarrow{\eta \otimes id} \mathcal{A}(m-1) \otimes \mathcal{A}(s) \xrightarrow{\tensor[_i]{\circ}{_j}} \mathcal{A}(s).
		\]
		
		\item if $j=\left(2,4,\ldots,2\lfloor \frac{m-1}{2} \rfloor\right)$, for any $r \geq 0$ and $i$ an $\lfloor \frac{m}{2} \rfloor$-tuple of elements in $\{1,\ldots,r\}$, the following composite is the identity
		\[
			\mathcal{A}(r) \xrightarrow{id \otimes \eta} \mathcal{A}(r) \otimes \mathcal{A}(m-1) \xrightarrow{\tensor[_i]{\circ}{_j}} \mathcal{A}(r).
		\]
		
		\item for $r,s,t \geq 0$, $i$ an $\lfloor \frac{m}{2} \rfloor$-tuple of elements in $\{1,\ldots,r\}$, $j$ and $j'$ respectively $\lfloor \frac{m}{2} \rfloor$-tuple and $\lfloor \frac{m-1}{2} \rfloor$-tuple of elements in $\{1,\ldots,s\}$, such that $j$ and $j'$ do not have any element in common, and $k$ an $\lfloor \frac{m-1}{2} \rfloor$-tuple of elements in $\{1,\ldots,t\}$, the following square commutes
		\[
			\xymatrix{
				\mathcal{A}(r) \otimes \mathcal{A}(s) \otimes \mathcal{A}(t) \ar[r]^-{id \otimes \tensor[_j]{\circ}{_k}} \ar[d]_{\tensor[_i]{\circ}{_{j'}} \otimes id} & \mathcal{A}(r) \otimes \mathcal{A}(s \vee t) \ar[d]^{\tensor[_i]{\circ}{_{k'}}} \\
				\mathcal{A}(r \vee s) \otimes \mathcal{A}(t) \ar[r]_{\tensor[_{i'}]{\circ}{_k}} & \mathcal{A}(r+s+t-2m+2)
			}
		\]
		
		\item if $m$ is even, for $r,s,t \geq 0$, $i$ and $i'$ two $\frac{m}{2}$-tuples of elements in $\{1,\ldots,r\}$, $j$ an $\frac{m}{2}-1$-tuple of elements in $\{1,\ldots,s\}$ and $k$ an $\frac{m}{2}-1$-tuple of elements in $\{1,\ldots,t\}$, the following diagram commutes
		\[
			\xymatrix{
				\mathcal{A}(r) \otimes \mathcal{A}(s) \otimes \mathcal{A}(t) \ar[r]^-{\tensor[_i]{\circ}{_j} \otimes id} \ar[d]_\simeq & \mathcal{A}(r \vee s) \otimes \mathcal{A}(t) \ar[dd]^{\tensor[_{j'}]{\circ}{_k}} \\
				\mathcal{A}(r) \otimes \mathcal{A}(t) \otimes \mathcal{A}(s) \ar[d]_{\tensor[_{i'}]{\circ}{_k} \otimes id} \\
				\mathcal{A}(r \vee t) \otimes \mathcal{A}(s) \ar[r]_-{\tensor[_{k'}]{\circ}{_j}} & \mathcal{A}(r \vee s \vee t)
			}
		\]
	\end{itemize}
\end{definition}

\begin{remark}
	An operad of complexity $1$ is a monoid in the category of collections, with monoidal product given by the Day convolution
	\[
		\mathcal{A} \otimes \mathcal{B}(n) = \coprod_{r+s=n} \mathcal{A}(r) \otimes \mathcal{B}(s).
	\]
	An operad of complexity $2$ is a non-symmetric operad as defined by Markl \cite[Definition 11]{markl2008operads}.
\end{remark}


\subsection{Join of operations}

The objective now is to describe the symmetric coloured operad whose algebras are operads of complexity $m$.

\begin{definition}\label{definitiongrafting}
	If an operation $x$ of $\mathcal{L}$ is the contatenation of $m+1$ substrings
	\[
		x_0 x_1 \ldots x_m,
	\]
	then we will say that $x$ is the \emph{join} of the string
	\begin{equation}\label{string1}
		x_{even} = x_0 | x_2 | \ldots | x_{2\lfloor \frac{m}{2} \rfloor}
	\end{equation}
	and the string
	\begin{equation}\label{string2}
		x_{odd} = x_1 | x_3 | \ldots | x_{2\lfloor \frac{m-1}{2} \rfloor+1},
	\end{equation}
	where we assume that the strings $x_{even}$ and $x_{odd}$ do not have any common integer.
%
\end{definition}

\begin{example}
	The operation of $\mathcal{L}_3$ given by
	\[
		\underbracket{1|1212}_{x_0}\underbracket{3|34|4}_{x_1}\underbracket{2|2}_{x_2}\underbracket{4|43}_{x_3}
	\]
	is the join of
	\[
		\underbracket{1|1212}_{x_0}{\color{black}|}\underbracket{2|2}_{x_2}
	\]
	and
	\[
		\underbracket{3|34|4}_{x_1}{\color{black}|}\underbracket{4|43}_{x_3}.
	\]
\end{example}

\begin{remark}\label{remarkjoin}
	When $m=0$, we can not construct the join of two strings. When $m=1$, a string is the join of $x_0$ and $x_1$ if it is the concatenation of them. When $m=2$, the string $x_0 x_1 x_2$ is the join of $x_0 | x_2$ and $x_1$. Note that if the strings are seen as planar trees through the correspondence of Remark \ref{remarkcorrespondence}, then $x_0 x_1 x_2$ is the planar tree obtained by grafting the leaf of $x_0 | x_2$, corresponding to the vertical bar between $x_0$ and $x_2$, to the root of $x_1$.
\end{remark}

\subsection{Operad for operads of complexity $m$}

%

\begin{definition}
	Let $\hat{\mathcal{L}}_m$ be the suboperad of $\mathcal{L}$ whose
	\begin{itemize}
		\item nullary operation is the string with $m-1$ vertical bars,
		\item unary operations are the identities,
	\end{itemize}
	and such that if two operations are in $\hat{\mathcal{L}}_m$, so is any join of them.
\end{definition}

%
%

%

\begin{theorem}
	$\hat{\mathcal{L}}_m$ is the operad for operads of complexity $m$.
\end{theorem}

\begin{proof}
	Let $\mathcal{A}$ be an algebra of $\hat{\mathcal{L}}_m$. The string with $m-1$ vertical bars induces the map \ref{equationunitmap}. Let $r,s \geq 0$, $i$ an $\lfloor \frac{m}{2} \rfloor$-tuple of elements in $\{1,\ldots,r\}$ and $j$ an $\lfloor \frac{m-1}{2} \rfloor$-tuple of elements in $\{1,\ldots,s\}$. Let $x_0|x_2|\ldots$ be the string $1|\ldots|1$ which starts with $1$, ends with $1$ and alternates between $1$ and the vertical bar $|$, such that in total there are $r$ vertical bars. The vertical bars between $x_{2i}$ and $x_{2i+2}$ correspond to the $\lfloor \frac{m}{2} \rfloor$-tuple $i$. Similarly, let $x_1|x_3|\ldots$ be the string $2|\ldots|2$ such that in total there are $s$ vertical bars and the bars between $x_{2i-1}$ and $x_{2i+1}$ correspond to the $\lfloor \frac{m-1}{2} \rfloor$-tuple $j$. The join of these two strings induces the map \ref{equationmultiplicationmap}. We leave it to the reader to check that the unit and associativity axioms are satisfied.
	
	Now let us assume that $\mathcal{A}$ is an operad of complexity $m$. The map induced by the string with $m-1$ vertical bars is given by the unit map \ref{equationunitmap}. The maps induced by the identities should be the identity maps. Now let us assume by induction that we have defined the maps induced by $x \in \hat{\mathcal{L}}_m(r_1,\ldots,r_k;r)$ and $y \in \hat{\mathcal{L}}_m(s_1,\ldots,s_l;s)$. Note that the string $x \circ y$ has the same integers as $x$ and $y$ together. The number of vertical bars of $x \circ y$ is the number of vertical bars of $x$ plus the number of vertical bars of $y$ minus $m-1$, that is $r \vee s$. Moreover, the way to join $x$ and $y$ gives an $\lfloor \frac{m}{2} \rfloor$-tuple $i$ of elements in $\{1,\ldots,r\}$ and an $\lfloor \frac{m-1}{2} \rfloor$-tuple $j$ of elements in $\{1,\ldots,s\}$. We can therefore define the map induced by $x \circ y$ as the composite
	\[
		\mathcal{A}(r_1) \otimes \ldots \otimes \mathcal{A}(r_k) \otimes \mathcal{A}(s_1) \otimes \ldots \otimes \mathcal{A}(s_l) \xrightarrow{f \otimes g} \mathcal{A}(r) \otimes \mathcal{A}(s) \xrightarrow{\tensor[_i]{\circ}{_j}} \mathcal{A}(r \vee s),
	\]
	where $f$ is the map induced by $x$ and $g$ is the map induced by $y$. 
	The associativity axioms of Definition \ref{definitionoperadofcomplexitym} ensure us that we get a well-defined functor. 
	We have therefore described two functors inverse of each other. This concludes the proof.
\end{proof}

\section{Bimodules over operads of complexity $m$}\label{sectionbimodules}

The objective of this section is to define a notion of $\mathcal{A}-\mathcal{B}$-bimodules, for $\mathcal{A}$ and $\mathcal{B}$ two operads of complexity $m$. When $m=1$ and $m=2$, we will recover the classical notions of bimodules over monoids for the Day convolution product, and over non-symmetric operads, respectively.

\subsection{Boardman-Vogt tensor product} For any coloured operad $\mathcal{P}$, we will write $\mathcal{P}_{\cdot \to \cdot \leftarrow \cdot}$ for the coloured operad whose
\begin{itemize}
	\item colours are pairs $(i,l)$ where $i$ is a colour of $\mathcal{P}$ and $l \in \{A,B,C\}$,
	\item operations are operations $p \in \mathcal{P}$ together with a label $l \in \{A,B,C\}$ for each source as well as for the target of $p$, with the condition that if the target has label $A$ (resp. $B$), then all the sources also have label $A$ (resp. $B$).
\end{itemize}
This operad is actually the Boardman-Vogt tensor product of $\mathcal{P}$ and the category $\cdot \to \cdot \leftarrow \cdot$, seen as a coloured operad. Its algebras are cospans $\mathcal{A} \to \mathcal{C} \leftarrow \mathcal{B}$ of algebras in $\mathcal{P}$. The cospan maps are induced by the identities where the target label is $C$ and the unique source is labelled with $A$ or $B$.

\subsection{Underlying graph of an operation}

For an operation $x \in \mathcal{L}(n_1,\ldots,n_k;n)$ and $1 \leq i < j \leq k$, we write $i <_m j$ if $c_{ij}(x) = m$ and $i$ appears before $j$ in the string $x$. Note that this relation is not always transitive. Recall that a \emph{directed graph} is given by a set of vertices $V$ and a set of edges from $v$ to $w$ for all $v,w \in V$.

\begin{definition}
	Let us assume that $m \geq 0$ is fixed and let $x \in \mathcal{L}_m(n_1,\ldots,n_k;n)$. We define $G_x$ as the directed graph with set of vertices $\{1,\ldots,k\}$ and with an edge from $i$ to $j$ if $i <_m j$. $G_x$ will be called the \emph{underlying graph} of $x$.
\end{definition}

\begin{example}
	The graph corresponding to the operation $12321434$ in $\mathcal{L}_3$ is the following:
	\[
	\xymatrix@R=1pc{
		1 \ar[rd] \\
		& 3 \ar[r] & 4 \\
		2 \ar[ru]
	}
	\]
\end{example}

\subsection{Bimodules and pointed bimodules}

\begin{definition}
	A directed graph $G$ is \emph{properly labelled} if each vertex come with a label in $\{A,B,C\}$ such that every time there is an edge from $v$ to $w$, $v$ has label $A$ or $w$ has label $B$.
\end{definition}

\begin{definition}
	Let $\mathcal{P}$ be a suboperad of $\mathcal{L}_m$. We define $\mathcal{P}_{\cdot + \cdot}$ as the suboperad of $\mathcal{P}_{\cdot \to \cdot \leftarrow \cdot}$ whose operations are such that their underlying graph is properly labelled. We define $\mathcal{P}_{\cdot \circ \cdot}$ as the suboperad of operations such that there is no way to add $m$ times an extra integer in the string corresponding to this operation, and label this integer with $B$, such that the resulting operation is still in $\mathcal{P}_{\cdot + \cdot}$.
\end{definition}

\begin{remark}\label{remarkcasetwo}
	The operations of $(\hat{\mathcal{L}}_2)_{\cdot + \cdot}$ are planar trees whose vertices are labelled with $A$, $B$ or $C$, where, for any pair of adjacent vertices, the vertex below has label $A$ or the one above has label $B$. Therefore, as it was established in the proof of \cite[Theorem 10.2]{batanindeleger}, the algebras for this operad are triple $(\mathcal{A},\mathcal{B},\mathcal{C})$, where $\mathcal{A}$ and $\mathcal{B}$ are two non-symmetric operads and $\mathcal{C}$ is a \emph{$1$-pointed $\mathcal{A}-\mathcal{B}$-bimodule}. Adding $2$ times an extra integer in a string corresponds to adding a unary vertex under the correspondence \ref{remarkcorrespondence}. Therefore, the operations of $(\hat{\mathcal{L}}_2)_{\cdot \circ \cdot}$ are planar trees as described above, but where the vertices labelled with $C$ \emph{lie on a line}. The algebra for this operads are triples $(\mathcal{A},\mathcal{B},\mathcal{C})$, where $\mathcal{A}$ and $\mathcal{B}$ are two non-symmetric operads and $\mathcal{C}$ is an $\mathcal{A}-\mathcal{B}$-bimodule.
\end{remark}

The previous remark motivates the following definition:

\begin{definition}
	Let $\mathcal{A}$ and $\mathcal{B}$ be two operads of complexity $m$. If the triple $(\mathcal{A},\mathcal{B},\mathcal{C})$ is an algebra of $(\hat{\mathcal{L}}_m)_{\cdot + \cdot}$, then $\mathcal{C}$ is called a \emph{pointed $\mathcal{A}-\mathcal{B}$-bimodule}. If the triple $(\mathcal{A},\mathcal{B},\mathcal{C})$ is an algebra of $(\hat{\mathcal{L}}_m)_{\cdot \circ \cdot}$, then $\mathcal{C}$ is called an \emph{$\mathcal{A}-\mathcal{B}$-bimodule}.
\end{definition}

\section{Fibration sequence theorem}\label{sectionfibrationsequences}

The objective of this section is to prove Theorem \ref{theorem}.

\subsection{Category of proper labellings}

In this subsection we define, for a directed graph $G$, the category of proper labellings of $G$, and we prove that if $G$ has no cycles, then this category has a contractible nerve. This will be useful later, to prove our cofinality theorem.

\begin{definition}
	For a directed graph $G$, let $\mathbb{C}(G)$ be the category whose objects are proper labellings of $G$, the morphisms change the labels of the vertices from $A$ or $B$ to $C$.
\end{definition}

\begin{example}
	If $G$ is the graph $\cdot \leftarrow \cdot \rightarrow \cdot$, then $\mathbb{C}(G)$ looks as in the picture below, which was also pictured in \cite[Figure 10.4]{batanindeleger}:
	\[
	\xymatrix@R=1pc@C=1pc{
		&&&&&& AAB \ar[ld] \ar[rd] \ar[dd] \\
		&&&&& CAB \ar[rd] && AAC \ar[ld] \\
		BBB \ar[rr] && BCB && BAB \ar[ll] \ar[ru] \ar[rd] \ar[rr] && CAC && AAA \ar[lu] \ar[ld] \ar[ll] \\
		&&&&& BAC \ar[ru] && CAA \ar[lu] \\
		&&&&&& BAA \ar[lu] \ar[ru] \ar[uu]
	}
	\]
\end{example}

\begin{lemma}\label{lemmacategoryofproperlabellingsgraphnocycles}
	If $G$ has no cycles, then $\mathbb{C}(G)$ has a contractible nerve.
\end{lemma}

\begin{proof}
	We proceed by induction on the number of vertices of $G$. If $G$ has no vertices, then $\mathbb{C}(G)$ is the category with one object and only the identity morphism. Otherwise, let $v$ be a vertex of $G$. Since $G$ has no cycles, we can assume that there are no edges having $v$ as a target. Let $\mathbb{C}_A(G)$ be the full subcategory of $\mathbb{C}(G)$ of labellings where $v$ has label $A$. Let $\mathbb{C}_B(G)$ be the full subcategory of $\mathbb{C}(G)$ of labellings where any vertex above $v$ has label $B$, where we say that a vertex $w$ is \emph{above} $v$ if there is a path $v = v_0 \to \ldots \to v_k = w$, $k \geq 1$.
	
	The category $\mathbb{C}_A(G)$ is isomorphic to the category $\mathbb{C}(G \setminus \{v\})$, where $G \setminus \{v\}$ is the graph obtained by removing the vertex $v$. The category $\mathbb{C}_B(G)$ is isomorphic to the category $\mathbb{C}(G_v)$, where $G_v$ is the graph obtained by removing all the vertices above $v$. The intersection of $\mathbb{C}_A(G)$ and $\mathbb{C}_B(G)$ is isomorphic to the category $\mathbb{C}(G_v \setminus \{v\})$. Therefore, by induction, the nerve of each of these categories is contractible. Since the union of $\mathbb{C}_A(G)$ and $\mathbb{C}_B(G)$ is $\mathbb{C}(G)$, this concludes the proof.
\end{proof}

\subsection{Cofinality}

Polynomial monads are known to be equivalent to coloured operad with a free action of the symmetric group \cite[Proposition 2.5.5]{kock}. The lattice path operad as well as the other coloured operads defined here are all polynomial monads. This fact will be implicit in what follows. Let $f: \mathcal{O} \to \mathcal{P}$ be a map of coloured operads. 

\begin{definition}
	For a colour $j$ of $\mathcal{P}$, let $f/j$ be the category whose
	\begin{itemize}
		\item objects are given by operations $p \in \mathcal{P}(f(i_1),\ldots,f(i_k);j)$, where $(i_1,\ldots,i_k)$ is a $k$-tuple of colours of $\mathcal{O}$,
		\item a morphism from $p \in \mathcal{P}(f(i_{11}),\ldots,f(i_{kl_k});j)$ to $q \in \mathcal{P}(f(i_1),\ldots,f(i_k);j)$ are given by $k$-tuples of operations
		\[
			o_1 \in \mathcal{O}(i_{11},\ldots,i_{1l_1};i_1),\ldots,o_k \in \mathcal{O}(i_{k1},\ldots,i_{kl_k};i_k)
		\]
		such that $m(q,f(o_1),\ldots,f(o_k)) = p$, where $m$ is the multiplication in $\mathcal{P}$.
	\end{itemize}
\end{definition}

It was proved \cite[Theorem 7.3]{batanin} that the categories of the definition above form a categorical algebra over $\mathcal{P}$ and has a universal property. This categorical algebra is written $\mathcal{Q}^\mathcal{P}$ and is called \emph{internal algebra classifier} induced by $f$. Following \cite[Definition 5.6]{batanindeleger}, we will say that $f$ is \emph{homotopically cofinal} if $\mathcal{Q}^\mathcal{P}$ has a contractible nerve.

\begin{theorem}\label{theoremcofinality}
	For any $m \geq 1$ and any suboperad $\mathcal{P}$ of $\mathcal{L}_m$, the map
	\[
		f: \mathcal{P}_{\cdot + \cdot} \to \mathcal{P}_{\cdot \to \cdot \leftarrow \cdot}
	\]
	given by inclusion of sets is homotopically cofinal.
\end{theorem}

\begin{proof}
	We follow the same strategy as in \cite{batanindeleger}. First, there is a commutative square in the category of coloured operads
	\[
		\xymatrix{
			\mathcal{P}_{\cdot + \cdot} \ar[r]^{pf} \ar[d]_f & \mathcal{P} \ar[d]^{id} \\
			\mathcal{P}_{\cdot \to \cdot \leftarrow \cdot} \ar[r]_-p & \mathcal{P}
		}
	\]
	where $p$ is the projection. According to \cite[Proposition 4.7]{batanindeleger}, this square induced a morphism of categorical algebras
	\begin{equation}\label{equationinducedmapbetweenclassifiers}
		\mathcal{P}_{\cdot \to \cdot \leftarrow \cdot}^{\mathcal{P}_{\cdot + \cdot}} \to p^* \left( \mathcal{P}^\mathcal{P} \right)
	\end{equation}
	To simplify the notations, let us assume that a colour $(i,l)$ of $\mathcal{P}_{\cdot \to \cdot \leftarrow \cdot}$ is fixed and let $F: \mathcal{X} \to \mathcal{Y}$ be the underlying functor of \ref{equationinducedmapbetweenclassifiers} at this index. We want to prove that $F$ induces a weak equivalence between nerves. This would give us the desired result since for each colour $\mathcal{Y}$ has a terminal object, so its nerve is contractible. According to Quillen Theorem A, we need to prove that for all $y \in \mathcal{Y}$, the nerve of $y/F$ is contractible. To achieve this result, we will prove that $F$ is smooth \cite[5.3.1]{cisinski}, that is, for all $y \in \mathcal{Y}$, the canonical inclusion functor
	\[
		F_y \to y/F,
	\]
	where $F_y$ is the fibre over $y$, induces a weak equivalence between nerve. According to \cite[Proposition 5.3.4]{cisinski}, it is equivalent to proving that for all morphisms $f_1: y_0 \to F(x_1)$ in $\mathcal{Y}$, the nerve of the \emph{lifting category} of $f_1$ is contractible. By definition, the objects of this lifting category are morphisms $f:x \to x_1$ in $\mathcal{X}$ such that $F(f) = f_1$, and morphisms are commutative triangles
	\[
		\xymatrix{
			x \ar[rr]^g \ar[rd]_f && x' \ar[ld]^{f'} \\
			& x_1
		}
	\]
	with $g$ a morphism in $F_{y_0}$.
	
	We will now describe the functor $F$ more explicitly. The objects of $\mathcal{Y}$ are operations of $\mathcal{P}$ with target $i$. A morphism from $p$ to $q$ is a $k$-tuple of operations $o_1,\ldots,o_k$ in $\mathcal{P}$ such that $m(q,o_1,\ldots,o_k)=p$, where $m$ is the multiplication and $k$ is the number of sources of $q$. For example, going back to Remark \ref{remarkcompositionlatticepathoperad}, there is a morphism
	\[
		213|13455|54423 \to 1|12|21
	\]
	The objects of $\mathcal{X}$ are objects of $\mathcal{Y}$ equipped with labels in $\{A,B,C\}$ for the sources and for the target. The description of the morphisms is the same as for $\mathcal{Y}$ but $o_1,\ldots,o_k$ should also be equipped with labels, such that the underlying graphs are properly labelled. Obviously, the labels of $p$ should correspond with the labels of $q,o_1,\ldots,o_k$. The functor $F$ forgets all the labels. For an object of $\mathcal{X}$, given by an operation of $\mathcal{P}$, the fibre of $F$ over this object is the category whose objects are all the possible labellings of the sources of this operation. The morphisms change labels from $A$ or $B$ to $C$. So, the fibre has the shape of a cube, and its nerve is contractible. Finally, let $f_1: y_0 \to F(x_1)$ be a morphism in $\mathcal{Y}$. It is given by composable operations $q,o_1,\ldots,o_k$ in $\mathcal{P}$ such that $q$ is equipped with a labelling. The lifting category of $f_1$ is the category $\mathbb{C}(G)$, where $G$ is the disjoint union of the underlying graphs of $o_1,\ldots,o_k$. $G$ has no cycles since the order of apparition of integers in a string is a linear order. Therefore, the nerve of $\mathbb{C}(G)$ is contractible according to Lemma \ref{lemmacategoryofproperlabellingsgraphnocycles}. This concludes the proof.
\end{proof}

%
%
%
%
%
%

\subsection{Proof of the theorem}

%

From now, we assume that operads of complexity $m$ are defined in a model category which is \emph{strongly $h$-monoidal} \cite[Definition 1.11]{bataninberger}.

\begin{lemma}\label{lemmaleftproper}
	The category of operads of complexity $m$ is left proper.
\end{lemma}

\begin{proof}
	We want to prove that $\hat{\mathcal{L}}_m$ is \emph{tame} \cite[Definition 6.19]{bataninberger}. This will give us the desired result thanks to \cite[Theorem 8.1]{bataninberger}. By definition, a coloured operad $\mathcal{P}$ is tame if $\mathcal{P}^{\mathcal{P}+1}$ is a coproduct of categories with terminal object. The coloured operad $\mathcal{P}+1$ is the operad for \emph{semi-free coproduct of $\mathcal{P}$-algebras}, that is, coproducts $X \amalg \mathcal{F}_{\mathcal{P}}(K)$ in the category of $\mathcal{P}$-algebras, where $\mathcal{F}_{\mathcal{P}}$ is the free $\mathcal{P}$-algebras functor. According to \cite[Section 6.20]{bataninberger}, when $\mathcal{P}=\hat{\mathcal{L}}_m$, $\mathcal{P}^{\mathcal{P}+1}$ is the category whose objects are operations of $\hat{\mathcal{L}}_m$ with each vertex of the underlying graph labelled with $X$ or $K$. A morphism is as described in the proof of Theorem \ref{theoremcofinality}, except that if a vertex is labelled with $K$, it should remain stable. We will now characterise the terminal objects of each connected component of this category. We say that two vertices $v$ and $w$ in a directed graph are \emph{adjacent} if there is an edge from $v$ to $w$ but there are no vertices $u$ together with an edge from $v$ to $u$ and an edge from $u$ to $w$. We say that a vertex is a \emph{root vertex} if there is no edge to it and a \emph{leaf vertex} if there is no edge from it. Just like in \cite[Section 9.2]{bataninberger}, the terminal objects are the operations such that two adjacent vertices in the underlying graph have different labels $X$ and $K$, and the root and leaf vertices have label $X$.
\end{proof}

\begin{lemma}\label{lemmapointedbimodules}
	Let $\mathcal{A},\mathcal{B}$ be two operads of complexity $m$. If $m \geq 3$, a pointed $\mathcal{A}-\mathcal{B}$-bimodule is the same as an $\mathcal{A}-\mathcal{B}$-bimodule $\mathcal{C}$ equipped with a map $1 \to \mathcal{C}_{m-1}$.
\end{lemma}

\begin{proof}
	If $m=2$, the result is obvious from Remark \ref{remarkcasetwo}. The argument is very similar for the case $m=1$. We will prove the case $m=3$. We need to prove that the coloured operad $(\hat{\mathcal{L}}_3)_{\cdot + \cdot}$ is the same as the coloured operad freely generated by adding the string with $2$ bars and target label $C$ to $(\hat{\mathcal{L}}_3)_{\cdot \circ \cdot}$. Since $(\hat{\mathcal{L}}_3)_{\cdot + \cdot}$ contains the string with $2$ bars and target label $C$ and also $(\hat{\mathcal{L}}_3)_{\cdot \circ \cdot}$ as a suboperad, the inclusion in the first direction is obvious. For the other direction, we need to prove that for any operation of $(\hat{\mathcal{L}}_3)_{\cdot + \cdot}$ given by a string $x$, there is a unique operation of $(\hat{\mathcal{L}}_3)_{\cdot \circ \cdot}$ given by a string $x'$, obtained from $x$ by adding several times $3$ occurrences of an extra integer labelled with $C$. Let us first assume $x$ is an identity. If the unique source of $x$ is labelled with $C$, then $x'$ is given by $x$. If it is labelled with $A$, we need to add $3$ occurrences of an extra integer labelled with $C$ as many times as $x$ has vertical bars. For example, if $x = 1|1|1|1$, then
	\[
		x' = 12|213|314|41432.
	\]
	Similarly, if the unique source is labelled with $B$, then
	\[
		x' = 23414|413|312|21.
	\]
	In both cases, $x'$ corresponds to an operation of $(\hat{\mathcal{L}}_3)_{\cdot \circ \cdot}$ because we can not add $3$ occurrences of an extra integer anymore. By induction, if an operation of $(\hat{\mathcal{L}}_3)_{\cdot + \cdot}$ is given by a string $z$ obtained as the join of two operations $x$ and $y$, then $z'$ will be given by the join of $x'$ and $y'$. Indeed, if it was possible to add $3$ occurrences of an extra integer in $z'$, it would also be possible in $x'$ or $y'$, which is a contradiction.
%
%
\end{proof}


Let $\mathrm{Operad}(m)$ be the category of operads of complexity $m$ and $\zeta$ be the terminal object in this category. $\mathcal{O} \in \mathrm{Operad}(m)$ is called \emph{multiplicative} when it is equipped with a map $\zeta \to \mathcal{O}$. Let $\mathrm{Bimod}(m)$ be the category of $\zeta-\zeta$-bimodules.

\begin{theorem}\label{theorem}
	Let $\mathcal{O}$ be a multiplicative operad of complexity $m$. If $m \leq 3$, there is a fibration sequence
	\[
		\Omega \mathrm{Map}_{\mathrm{Operad}(m)} (\zeta,u^*(\mathcal{O})) \to \mathrm{Map}_{\mathrm{Bimod}(m)} (\zeta,v^*(\mathcal{O})) \to \mathcal{O}(m-1),
	\]
	where $u^*$ and $v^*$ are the appropriate forgetful functors.
\end{theorem}

\begin{proof}
	According to Lemma \ref{lemmaleftproper}, $\mathrm{Operad}(m)$ is left proper. Therefore, using \cite[Theorem 4.5]{deleger}, we have the delooping
	\begin{equation}\label{equation1}
		\Omega \mathrm{Map}_{\mathrm{Operad}(m)} (\zeta,u^*(\mathcal{O})) \to \mathrm{Map}_{S^0 / \mathrm{Operad}(m)} (\zeta,g^*(\mathcal{O})),
	\end{equation}
	where $S^0 := \zeta \amalg \zeta$ and $g^*$ is induce by the unique map $S^0 \to \zeta$.
	
	Let $\mathrm{PBimod}(m)$ be the category of pointed $\zeta-\zeta$-bimodules. There is a forgetful functor $h^*: S^0 / \mathrm{Operad}(m) \to \mathrm{PBimod}(m)$. This forgetful functor has a left adjoint $h_!$ forming a Quillen adjunction. This Quillen adjunction is actually induced by a map of coloured operads. We want to prove that $h_!$ is a \emph{left cofinal Quillen functor} \cite[Definition 4.7]{deleger}. Let $\mathbb{B}$ be the category of pairs of operads of complexity $m$. Let $\Phi: \mathbb{B}^{op} \to \mathrm{CAT}$ and $\Psi: \mathbb{B}^{op} \to \mathrm{CAT}$ be the functors which sends a pair $(\mathcal{A},\mathcal{B})$ to the category of pointed $\mathcal{A}-\mathcal{B}$ bimodules and of cospans $\mathcal{A} \to \mathcal{C} \leftarrow \mathcal{B}$ of operads of complexity $m$, respectively. There is a Quillen adjunction between the Grothendieck constructions over these functors. This Quillen adjunction is induced by the operad map of Theorem \ref{theoremcofinality}, when we take $\mathcal{P} = \hat{\mathcal{L}}_m$. According to \cite[Remark 4.8]{deleger}, the left adjoint of this adjunction is a left cofinal Quillen functor. We deduce from \cite[Theorem 3.33]{bwdgrothendieck} that $h_!$ is also a left cofinal Quillen functor. Therefore, according to \cite[Lemma 4.9]{deleger}, there is a weak equivalence
	\begin{equation}\label{equation2}
		\mathrm{Map}_{S^0 / \mathrm{Operad}(m)} (\zeta,h^*(\mathcal{O})) \to \mathrm{Map}_{S^0 / \mathrm{PBimod}(m)} (\zeta,h^*g^*(\mathcal{O})).
	\end{equation}
	
	Finally, we deduce from Lemma \ref{lemmapointedbimodules} and \cite[Theorem 4.13]{deleger} that there is a fibration sequence
	\begin{equation}\label{equation3}
		\mathrm{Map}_{\mathrm{PBimod}(m)} (\zeta,h^*g^*(\mathcal{O})) \to \mathrm{Map}_{\mathrm{Bimod}(m)} (\zeta,v^*(\mathcal{O})) \to \mathcal{O}(m-1).
	\end{equation}
	We get the desired result by combining \ref{equation1}, \ref{equation2} and \ref{equation3}.
\end{proof}

%
%
%
%
%
%
%

\bibliographystyle{plain}
\bibliography{latticepath}

\end{document}